\documentclass{amsart}
\usepackage{amscd,amssymb,amsopn,amsmath,amsthm,graphics,amsfonts,accents,enumerate,verbatim,calc}
\usepackage[dvips]{graphicx}
\usepackage{pgfplots}
\usepackage[pagebackref,colorlinks=true,linkcolor=red,citecolor=blue]{hyperref}
\hypersetup{pdfencoding=auto}
\usepackage[all]{xy}
\usepackage{cite}
\usepackage{tikz}
\usepackage{tikz-cd}
\usepackage{mathrsfs}

\usepackage[T1]{fontenc}
\usepackage{newtxtext, newtxmath}
\usepackage[utf8]{inputenc} % For Unicode input
\usepackage{textcomp}

\usepackage{microtype}

\calclayout
\newcommand{\rt}{\rightarrow}
\newcommand{\lrt}{\longrightarrow}

\newcommand{\st}{\stackrel}

\newcommand{\La}{\Lambda}

\newcommand{\Z}{\mathbb{Z} }

\newcommand{\CA}{\mathcal{A} }
\newcommand{\CC}{\mathcal{C} }

\newcommand{\CI}{\mathcal{I} }

\newcommand{\CP}{\mathcal{P} }

\newcommand{\CR}{\mathcal{R} }

\newcommand{\CX}{\mathcal{X} }

\newcommand{\CV}{\mathcal{V}}
\newcommand{\CU}{\mathcal{U}}
\newcommand{\CZ}{\mathcal{Z} }

\newcommand{\CM}{\mathcal{M} }

\newcommand{\Mod}{{\rm{Mod\mbox{-}}}}

\newcommand{\mmod}{{\rm{{mod\mbox{-}}}}}

\newcommand{\prj}{{\rm{prj}\mbox{-}}}

\newcommand{\Gprj}{{\Gp\mbox{-}}}

\newcommand{\inj}{{\rm{inj}\mbox{-}}}

\newcommand{\supp} {{\rm supp}\mbox{-}}
\newcommand{\ind}{{\rm ind}\mbox{-}}

\newcommand{\im}{{\rm{Im}}}

\newcommand{\add}{{\rm{add}\mbox{-}}}

\newcommand{\Gp}{{\mathcal{G}p}}

\newcommand{\Coker}{{\rm{Coker}}}
\newcommand{\Ker}{{\rm{Ker}}}

\newcommand{\Hom}{{\rm{Hom}}}
\newcommand{\Ext}{{\rm{Ext}}}

\newtheorem{theorem}{Theorem}[section]
\newtheorem{corollary}[theorem]{Corollary}
\newtheorem{lemma}[theorem]{Lemma}

\theoremstyle{definition}
\newtheorem{definition}[theorem]{Definition}

\newtheorem{remark}[theorem]{Remark}
\newtheorem{s}[theorem]{}

\theoremstyle{plain}
\newtheorem{stheorem}{Theorem}[subsection]%subsection --> section
\newtheorem{scorollary}[stheorem]{Corollary}
\newtheorem{slemma}[stheorem]{Lemma}
\newtheorem{sproposition}[stheorem]{Proposition}

\theoremstyle{definition}
\newtheorem{sdefinition}[stheorem]{Definition}

\newtheorem{sremark}[stheorem]{Remark}

\def\k{\Bbbk}

\numberwithin{equation}{subsection}

\pgfplotsset{compat=1.18}
\begin{document}

\title[Covering techniques in higher AR-theory]{Covering techniques in higher Auslander-Reiten theory}

\author[Asadollahi, Hafezi, Sourani and Vahed]{Javad Asadollahi, Rasool Hafezi, Zohreh Sourani and Razieh Vahed}

\address{Department of Pure Mathematics, Faculty of Mathematics and Statistics, University of Isfahan, Isfahan 81746-73441, Iran}
\email{asadollahi@sci.ui.ac.ir}

\address{School of Mathematics and Statistics, Nanjing University of Information Science \& Technology, Nanjing, Jiangsu 210044, P.\,R. China}
\email{hafezi@nuist.edu.cn}

\address{Department of Pure Mathematics, Faculty of Mathematics and Statistics, University of Isfahan, Isfahan 81746-73441, Iran}
\email{z.sourani@sci.ui.ac.ir}

\address{Department of Mathematics, Khansar Campus, University of Isfahan, Iran and \\ School of Mathematics, Institute for Research in Fundamental Sciences (IPM), P.O. Box: 19395‐ 5746, Tehran, Iran.}
\email{r.vahed@khc.ui.ac.ir, vahed@ipm.ir }

\makeatletter \@namedef{subjclassname@2020}{\textup{2020} Mathematics Subject Classification} \makeatother



\begin{abstract}
This paper investigates the behavior of $n$-precluster tilting subcategories under the push-down functor in the context of Galois coverings of locally bounded categories. Building on higher Auslander-Reiten theory and covering techniques, we establish that for a locally support-finite category $\CC$ with a free group action $G$ on its indecomposables, the push-down functor maps $G$-equivariant $n$-precluster tilting subcategories of $\mmod\CC$ to $n$-precluster tilting subcategories of
$\mmod(\CC/G)$, and vice versa. These results provide a framework for studying $\tau_n$-selfinjective algebras. We further prove that $\mmod\CC$ is $n$-minimal Auslander-Gorenstein if and only if $\mmod(\CC/G)$ is so, under square-free conditions on $\CC/G$. Additionally, we analyze support $\tau_n$-tilting pairs via the push-down functor, showing that locally $\tau_n$-tilting finiteness is preserved under Galois coverings. Our work offers new insights into the interplay between higher homological algebra and covering theory in representation-finite contexts.
\end{abstract}

\maketitle

\tableofcontents

\section{Introduction}
Iyama introduces the concept of higher Auslander-Reiten theory in \cite{I} by examining a special subcategory of the module category known as $n$-cluster tilting subcategories, rather than considering the entire module category. He provides a higher version of Auslander’s correspondence for Artin algebras and related rings. In \cite{I2}, it is demonstrated that any finite $n$-cluster tilting subcategories correspond to $n$-Auslander algebras, which are defined as Artin algebras $\Gamma$ satisfying the condition:
\[{\rm domdim} \Gamma \geq n + 1 \geq {\rm gldim} \Gamma.\]

Here, ${\rm domdim}$ refers to the dominant dimension, and ${\rm gldim}$ denotes the global dimension. For further results in higher Auslander-Reiten theory, readers may refer to \cite{I, I2, I3}. \\ Additionally, an $n$-dimensional analogue of representation-finite hereditary algebras, called $n$-representation finite algebras, has been introduced and studied in \cite{DI}. An Artin algebra $\La$ is an $n$-representation finite algebra if there exists an $n$-cluster tilting $\La$-module.\\ Recently, Iyama and Solberg \cite{IS} have introduced the notions of $n$-precluster tilting subcategories and $\tau_n$-selfinjective algebras, which serve as generalizations of two seemingly different concepts: (finite) $n$-cluster tilting subcategories and $\tau$-selfinjective algebras. An Artin algebra $A$ is called $\tau_n$-selfinjective if it admits an $n$-precluster tilting module.

In \cite{IS}, a bijection between $n$-precluster tilting subcategories and $n$-minimal Auslander-Gorenstein algebras is established. This relationship serves as a higher-dimensional analogue of the Auslander-Solberg correspondence \cite{AS1}, as well as a Gorenstein analogue of the $n$-Auslander correspondence \cite{I2}. An Artin algebra $\Gamma$ is called an $n$-minimal Auslander-Gorenstein algebra if it satisfies the condition:
\[{\rm domdim} \Gamma \geq n + 1 \geq {\rm id} \Gamma,\]
where ${\rm id}$ denotes the injective dimension.

The covering technique has been introduced into the representation theory of Artin algebras and has been developed in a series of papers by K. Bongartz, P. Gabriel, C. Riedtmann \cite{BG, G, Rid}, and E. Green, A. de la Peña, and Martínez-Villa \cite{Gr, MD}, among others. This technique serves as an important tool for investigating the module category of an algebra. Specifically, it simplifies a problem related to the category of $ \La $-modules (for an algebra $ \La $) by reducing it to the category of modules over a simpler category $ \mathcal{C} $. This category often has a nice action by a group $ G $, allowing $ \La $ to be viewed as equivalent to the orbit category $ \mathcal{C}/G $.

One of the key results in this area, due to Gabriel \cite{G}, states that if $ \mathcal{C} $ is a locally bounded $ \k $-category (where $ \k $ is a field) and $ G $ is a group that acts freely on $ \text{ind} \mbox{-} \mathcal{C} $, then $ \mathcal{C} $ is locally representation-finite if and only if $ \mathcal{C}/G $ is also locally representation-finite. Here, $ \text{ind} \mbox{-} \mathcal{C} $ denotes the full subcategory of $ \text{mod}\mbox{-} \mathcal{C} $ whose objects provide a complete set of representatives for isoclasses of indecomposable modules in $ \text{mod} \, \mathcal{C} $, closed under the $ G $-action naturally induced from $ \mathcal{C} $ (for more details, see Section \ref{Pre}).

Gabriel's theorem has been further extended to a wide class of locally bounded categories, which includes locally representation-finite categories \cite{DLS, DS}. It has been demonstrated that if $ \mathcal{C} $ is a locally support-finite category with a free action of a group $ G $ on $ \text{ind} \, \mathcal{C} $, then $ \mathcal{C}/G $ is locally support-finite as well. Moreover, the Galois covering $ P: \mathcal{C} \to \mathcal{C}/G $ induces a bijection between the $ G $-orbits of isoclasses of indecomposable objects in $ \text{mod} \mbox{-} \mathcal{C} $ and the isoclasses of indecomposable objects in $ \text{mod} \mbox{-} (\mathcal{C}/G) $. For more information on this topic see e.g. \cite{DS2}, \cite{ As3}, \cite{ As4}, \cite{AHV}, \cite{ BL} and \cite{ BL2}.

Darp\"{o} and Iyama \cite{DI} utilized covering techniques to construct $n$-representation finite self-injective algebras as orbit algebras of the repetitive categories of certain algebras $\La$ with finite global dimension, referred to as $\nu_n$-finite algebras. It has been proven in \cite{DI} that the push-down functor maps every $ n $-cluster tilting subcategory of $ \text{mod} \mbox{-} \mathcal{C} $ (which meets certain conditions) to an $ n $-cluster tilting subcategory of $ \text{mod} \mbox{-} (\mathcal{C}/G) $, for a locally bounded category $ \mathcal{C} $ with a free action of a group $ G $ on $ \text{ind} \mbox{-}\mathcal{C} $. Using this, they have constructed $ n $-representation finite self-injective algebras as orbit algebras of the repetitive categories associated with $\nu_n $-finite algebras. Paquette, Prasad, and Wehlau \cite{PPW} have also examined the behavior of $ \tau $-tilting and $ \tau $-rigid subcategories of $ \text{mod} \mbox{-} \mathcal{C} $ under the push-down functor for sufficiently nice locally bounded $ G $-categories $ \mathcal{C} $.

In this paper, motivated by these results, we investigate the behavior of $ n $-precluster tilting subcategories under the push-down functor. We prove that if $ \CC $ is a locally support-finite category with a free action of a group $ G $ on $ \ind \CC $, then the push-down functor takes every $ G $-equivariant $ n $-precluster tilting subcategory of $ \mmod \CC $ to an $ n $-precluster tilting subcategory of $ \mmod(\CC/G)$ (see Theorem \ref{Main1}). Additionally, we show that if $ \CC $ is a locally bounded $ G $-category with a free $ G $-action on $ \ind \CC $, and $ \CV \subseteq P_{*}(\mmod \CC) $ is an $ n $-precluster tilting subcategory of $\mmod(\CC/G)$, then $ P_{*}^{-1}(\CV) = \{X \in \mmod \CC \mid P_{*}(X) \in \CV\} $ is a $ G $-equivariant $ n $-precluster tilting subcategory of $ \mmod \CC $ (see Theorem \ref{Main2}). These two theorems enable us to study the $ \tau_n $-selfinjectivity of Galois coverings over locally support-finite categories (see Corollary \ref{tau_nSelfinj}).

Moreover, following the work of \cite{IS}, we define an abelian category
$ \CA $ to be an $ n $-minimal Auslander-Gorenstein category if it satisfies the following conditions:
\begin{itemize}
	\item[$(i)$] ${\rm dom.dim } (\CA) \geq n+1$; that is, every projective object of $ \CA $ has an injective resolution whose first $ n $ terms are projective,
	\item[$(ii)$] every projective object has an injective dimension that is less than or equal to $ n+1 $.
\end{itemize}

An analog of the higher Auslander-Solberg correspondence for exact categories was presented and proved in \cite{Gre}. This result demonstrates that the map $\CU \mapsto \mmod \CU$ induces a bijective correspondence between equivalence classes of $n$-precluster tilting subcategories of $\mmod \CC$ (for some locally bounded $k$-category $\CC$) and equivalence classes of $n$-minimal Auslander-Gorenstein categories.

We prove that if $\CC$ is a locally bounded $G$-category with a free $G$-action on $\ind \CC$ such that $\CC/G$ is square-free, then $\mmod \CC$ is an $n$-minimal Auslander-Gorenstein category if and only if $\mmod(\CC/G)$ is also an $n$-minimal Auslander-Gorenstein category (see Proposition \ref{BonGab}). This result can be viewed as a generalization of Proposition 3.5 from \cite{BG}.

Let $\CC$ be a locally bounded $k$-category and $\CU$ be an $n$-precluster tilting subcategory of $\mmod \CC$. It can be shown that ${}^{\perp_{n-1}} \CU = \CU^{\perp_{n-1}}$, see Remark \ref{Remark3.12}. We define $\CZ(\CU) := {}^{\perp_{n-1}} \CU = \CU^{\perp_{n-1}}$. We demonstrate that there is an equivalence of categories between $\CZ(\CU)$ and $\Gprj \CU$, where $\Gprj \CU$ denotes the full subcategory of $\mmod \CU$ consisting of all Gorenstein projective objects (see Proposition \ref{equivalence}).

Using this result, we obtain that if $\CC$ is a locally support-finite $G$-category with a free $G$-action on $\ind \CC$, and $\CU$ is a $G$-equivariant $n$-precluster tilting subcategory of $\mmod \CC$, then the functor $\mmod P_*: \mmod \CU \to \mmod \CU/G$ serves as a Galois $G$-precovering of $n$-minimal Auslander-Gorenstein categories. This induces a Galois $G$-covering $\mmod P_*|: \Gprj \CU \to \Gprj \CU$, as discussed in Proposition \ref{Prop4.5}.

Zhou and Zhu in \cite{ZZ} introduced $\tau_n$-tilting objects for every $n$-cluster tilting subcategory of abelian category with enough projectives.
In Section \ref{tau_nTilting}, we examine support $(G, \tau_n)$-tilting pairs within the category $\mmod \CC$ and their image under the push-down functor for a locally bounded $G$-category $\CC$.

Let $\CC$ be such a category with a free $G$-action on $\ind \CC$. We consider an $n$-cluster tilting subcategory of $\mmod \CC$ such that $P_*(\CM)$ serves as an $n$-cluster tilting subcategory of $\mmod(\CC/G)$. Here, the notation $s(G, \tau_n)\mbox{-}\mathsf{tilt}(\CC)$ denotes the set of all support $(G, \tau_n)$-tilting $\CC$-modules in $\CM$, while $s \tau_n\mbox{-}\mathsf{tilt}(\CC/G)$ represents the set of all support $\tau_n$-tilting $\CC/G$-modules in $P_*(\CM)$ (see Definition \ref{t_n tilt}).

The push-down functor
\[P_*: s(G, \tau_n)\mbox{-}\mathsf{tilt}(\CC) \to s \tau_n\mbox{-}\mathsf{tilt}(\CC/G)\]
is a $G$-precovering. Furthermore, if $\CC$ is also locally support-finite, then $P_*$ becomes a $G$-covering map (see Proposition \ref{Prop-tau_n}).

This proposition allows us to establish a higher version of Theorem B from \cite{PPW}: A locally support-finite $G$-category $\CC$ with a free $G$-action on $\ind \CC$ is locally $(G, \tau_n)$-tilting finite if and only if $\CC/G$ is locally $\tau_n$-tilting finite (see Corollary \ref{Corr}).

Recall that a locally bounded $k$-category $\CC$ is called locally $(G, \tau_n)$-tilting finite, or locally $\tau_n$-tilting finite, if for every object $x$ in $\CC$, there exist only finitely many indecomposable $(G, \tau_n)$-rigid, or $\tau_n$-rigid, $\CC$-modules $M$ such that $M(x) \neq 0$.

\section{Preliminaries}\label{Pre}
In this paper, we denote by $\k$ a field. A category $\CC$ is referred to as a $\k$-linear category, or simply a $\k$-category, if for any two objects $x$ and $y$ in $\CC$, the set of morphisms $\CC(x,y)$ forms a $\k$-vector space. Additionally, the composition of morphisms in this category is required to be $\k$-bilinear.

\begin{definition}\label{LBB}({\cite[2.1]{BG}}) Let $\mathcal{C}$ be a $\k$-category. We say that $\mathcal{C}$ is locally bounded if it satisfies the following conditions:
\begin{itemize}
 \item[$(i)$] For each object $x \in \mathcal{C}$, the endomorphism algebra $\mathcal{C}(x,x)$ is local.
 \item[$(ii)$] Distinct objects in $\mathcal{C}$ are not isomorphic.
 \item[$(iii)$] For all objects $x, y \in \mathcal{C}$, the dimension of $\mathcal{C}(x,y)$ over $\k$ is finite; that is, $[\mathcal{C}(x,y) : \k] < \infty$.
 \item[$(iv)$] Each object $x \in \mathcal{C}$ has only finitely many objects $y \in \mathcal{C}$ such that either $\mathcal{C}(x,y) \neq 0$ or $\mathcal{C}(y,x) \neq 0$.
\end{itemize}
\end{definition}

\begin{definition}
Let $\CC$ be a locally bounded $\k$-category, and let $G$ be a group. We say that $\CC$ is a category with a $G$-action or a $G$-category if there is a group homomorphism $A: G \to \text{Aut}(\CC)$, where $\text{Aut}(\CC)$ is the group of $\k$-linear automorphisms of $\CC$.

A $\k$-linear functor $F$ is considered to belong to $\text{Aut}(\CC)$ if there exists a $\k$-linear functor $G: \CC \to \CC$ such that $F \circ G = 1$ and $G \circ F = 1$.

For convenience, for each $a \in G$ and each $x \in \CC$, we will denote the action of $a$ on $x$ as $ax$ instead of writing $A(a)(x)$. The action of $G$ on $\CC$ is called free if $ax \neq x$ for every object $x$ in $\CC$ and for every $a \neq 1$ in $G$.
\end{definition}

Throughout, $G$ denotes a group.

\begin{definition}\label{k-iso}
Let $\mathcal{C}$ and $\mathcal{C}'$ be locally bounded $\k$-categories, and let $\mathcal{C}$ be a $G$-category equipped with a free $G$-action. A $\k$-linear functor $F: \mathcal{C} \rightarrow \mathcal{C}'$ is called a Galois $G$-precovering with group $G$ if it satisfies the following conditions:
\begin{itemize}
\item[$(i)$] For every $a \in G$, $F \circ A(a) = F$.
\item[$(ii)$] For each $x \in \mathcal{C}$, $F^{-1}(Fx) = Gx$.
\item[$(iii)$] For each $x, y \in \mathcal{C}$, $F$ induces the following $\k$-module isomorphisms:
 \[ F^{(1)}_{y,x}: \bigoplus_{a\in G} \mathcal{C}(ax,y) \rightarrow \mathcal{C}'(F(x),F(y)), \]
 \[ F^{(2)}_{y,x}: \bigoplus_{a\in G} \mathcal{C}(x,ay) \rightarrow \mathcal{C}'(F(x),F(y)). \]
\end{itemize}
Additionally, if $F: {\rm Obj}(\mathcal{C}) \rightarrow {\rm Obj}(\mathcal{C'})$ is surjective, then $F$ is called a Galois $G$-covering.
\end{definition}

\begin{definition}
Let $\mathcal{C}$ be a $G$-category with a free $G$-action. The orbit category $\mathcal{C}/G$ of $\mathcal{C}$ by $G$ is defined as follows:
\begin{itemize}
\item[$(i)$] The class of objects consists of all orbits, denoted $Gx = \{ ax \mid a \in G \}$, for every object $x$ in $\mathcal{C}$.

 \item[$(ii)$] For each pair of objects $u, v \in \mathcal{C}/G$, the morphisms are defined as:
 \[\mathcal{C}/G(u,v) = \left\{ (f_{y,x})_{x \in u, y \in v} \in \prod_{x \in u, y \in v} \mathcal{C}(x,y) \ \bigg| \ af_{y,x} = f_{ay, ax} \ \text{for all } a \in G \right\}. \]

\item[$(iii)$] For two composable morphisms $u \st{f}\lrt v \st{g}\lrt w$ in $\mathcal{C}/G$, we define their composition as
 \[ gf = \left( \sum_{y \in v} g_{z,y} f_{y,x} \right)_{x \in u, z \in w}. \]
\end{itemize}
\end{definition}
\vspace{0.2cm}

There is a canonical functor $P : \CC \lrt \CC/G$ given by $P(x)=Gx$ and $P(f)= (\delta_{a , b} a f)_{(a , b) \in G \times G}$ for every $x, y \in \CC$ and for every $f \in \CC(x, y)$.
It is proved, in \cite{G}, that $P: \CC \lrt \CC/G$ is a Galois $G$-covering with group $G$.

\s\label{Functor}
We denote by $\Mod \CC$ the category of contravariant $\k$-linear functors from $\CC$ to $\Mod \k$, which is referred to as the category of (right) $\CC$-modules. Let $F$ and $G$ be objects in $\Mod \CC$. We denote the morphism set from $F$ to $G$ in $\Mod \CC$ as $\Hom_{\CC}(F,G)$.

For any object $x \in \CC$, we can apply the Yoneda lemma to show that $\CC(-,x)$ is a projective object in $\Mod \CC$. Furthermore, for any functor $F \in \Mod \CC$, there exists an epimorphism
$$\coprod_i \CC(-, x_i) \to F \to 0,$$
where $x_i$ ranges over a subclass of objects in $\CC$. Consequently, the collection $\{ \CC(-, x) \mid x \in \CC \}$ forms a generating set for $\Mod \CC$.

An object $F$ in $\Mod \CC$ is called finitely generated if there exists an integer $n$ and an epimorphism

\[\bigoplus_{i=1}^n \CC(-, x_i) \xrightarrow{\alpha} F \to 0, \]
for some objects $x_i$ in $\CC$. Additionally, if the kernel functor $\Ker \alpha$ is also finitely generated, then $F$ is referred to as finitely presented. The category of all finitely generated $\CC$-modules is denoted by $\mmod \CC$, and $ \text{ind} \mbox{-} \mathcal{C} $ denotes the full subcategory of $ \text{mod}\mbox{-} \mathcal{C} $ whose objects provide a complete set of representatives for isoclasses of indecomposable modules in $ \text{mod} \, \mathcal{C} $, closed under the $ G $-action naturally induced from $ \mathcal{C} $.

\begin{remark}\label{flength}
Let $\CC$ be a locally bounded $\k$-category and $ F $ be a finitely generated module over a category $ \CC $. Then $ F $ can be expressed as a quotient of a finite direct sum of the representable modules $ \CC(-,x) $, with the property that the sum $ \sum_{y \in \CC} [\CC(y,x): \k] $ is finite. Consequently, for every finitely generated $ \CC $-module $ F $, we have $ \sum_{x \in \CC} [F(x): \k] < \infty $. This implies that every finitely generated $ \CC $-module has finite length.

Furthermore, the converse is also true: if $ F $ is a $ \CC $-module satisfying $ \sum_{x \in \CC}[F(x): \k] < \infty $, then $ F $ is finitely generated. For more details, see \cite[2.2]{BG}.

Thus, a $ \CC $-module $ F $ is finitely generated if and only if it is finitely presented.
Additionally, the inequality $ \sum_{x \in \CC}[F(x): \k] < \infty $ for a finitely generated $ \CC $-module $ F $ indicates that $ F $ has finite support.
\end{remark}

\s \label{action}
Let $\CC$ be a $G$-category. The $G$-action on $\CC$ induces a $G$-action on $\Mod \CC$. Specifically, for each $a \in G$, we can define a $\k$-linear automorphism $\overline{A}(a): \Mod \CC \to \Mod \CC$ by
\[\overline{A}(a)(X) = X \circ A(a^{-1}),\]
for all $X \in \Mod \CC$. To simplify the notation, we denote ${}^a X := \overline{A}(a)(X)$, and for each morphism $\lambda: X \to Y$ in $\Mod \CC$, we write ${}^a \lambda$ in place of $\overline{A}(a)(\lambda)$. We say that $G$ acts freely on $\ind \CC$ if ${}^a M \ncong M$ for every $\CC$-module $M$ in $\ind \CC$ and for each $a \neq 1$ in $G$.

Note that for every $x \in \CC$, we have
\[{}^{a}\CC(-, x) = \CC(a^{-1}(-), x) \cong \CC(-, ax).\]

\s{\sc Push-down functor.}\label{push-down}
Let $\CC$ be a $G$-category with a free $G$-action. The canonical functor $P: \CC \lrt \CC/G$ induces a functor $P^{*}: \Mod (\CC/G) \lrt \Mod \CC$, given by $P^{*}(X) =X \circ P$, for every $X \in \Mod (\CC/G)$. This functor is called the pull-up of $P$. It is well known that $P^{*}$ possesses a left adjoint $P_{*}$, which is called the push-down of $P$. Indeed, $P_{*}: \Mod \CC \lrt \Mod (\CC/G)$ is defined as follows. Let $X \in \Mod \CC$, then for an object $u$ of $\CC/G$, $P_{*}(X)(u):= \oplus_{x \in u}X( x)$, and for a morphism $f:u \rt v$ in $\CC/G$, where $f=(f_{y,x })_{x \in u , y \in v}$, $(P_{*}X)(f)$ is given by the following commutative diagram
\[ \xymatrix@C-.5pc@R-.5pc{ (P_{*}X)(v) \ar[rr]^{(P_{*}X)(f)}\ar@{=}[d] & & (P_{*}X)(u)\ar@{=}[d] \\ \oplus_{y\in v} X( y) \ar[rr]^{(X(f_{y,x }))_{x,y }} & & \oplus_{x \in u}X( x).}\]
Also, if $\tau: X \lrt X'$ is a morphism in $\Mod \CC$, then for every $u \in {\rm Obj}(\CC/G)$, $(P_{*}\tau)_{u}$ is defined by the following commutative diagram
\[\xymatrix@C-.5pc@R-.5pc {(P_{*}X)(u) \ar[rr]^{(P_{*}\tau)_u}\ar@{=}[d] & & (P_{*}X')(u)\ar@{=}[d] \\ \oplus_{x \in u} X( x) \ar[rr]^{\oplus_{x \in u}\tau_{ x}} & & \oplus_{x \in u}X'( x).}\]
Note that for every $\CC$-module $X$ there is the canonical isomorphism $P^{*} P_{*} (X) \cong \oplus_{a \in G} {}^a X$, see \cite[Lemma 3.2]{G}.

$\CC$ is called locally representation-finite, provided that for each object $x$ of $\CC$ there are only finitely many modules $M$ in $\ind \CC$, up to isomorphism, with $M(x)\neq 0$. We need the following fundamental theorem which is due to Gabriel \cite{G}.

\begin{theorem} \cite[Proposition 3.1, Lemma 3.5, Theorem 3.6]{G} \label{Corres}
Let $\CC$ be a locally bounded $\k$-category and let $G$ be a group acting freely on $\ind \CC$. We can consider the following statements:

\begin{itemize}
\item[$(i)$] The push-down functor $P_{*}:\mmod \CC \to \mmod (\CC/G)$ serves as a Galois $G$-precovering. It maps any indecomposable $\CC$-module to an indecomposable $\CC/G$-module.

\item[$(ii)$] If $\CC$ is locally representation-finite, then there is a Galois $G$-precovering
$P_*: \ind \CC \lrt \ind (\CC/G)$. Indeed,
 the push-down functor induces a bijection between the set of $G$-orbits of isomorphism classes of modules in $\ind \CC$ and the set of isomorphism classes of modules in $\ind(\CC/G)$.
\end{itemize}
\end{theorem}

We recall that for a $\CC$-module $M$, the support of $M$, denoted as ${\supp} M$, refers to the full subcategory of $\CC$ consisting of all objects $x$ in $\CC$ such that $M(x) \neq 0$. Following \cite{DS2}, a locally bounded $\k$-category $\CC$ is termed locally support-finite if, for every object $x \in \CC$, the full subcategory formed by the points of all $\supp M$, where $M \in \ind \CC$ and $M(x) \neq 0$, is finite. Clearly, every finite category is locally support-finite.
By Remark \ref{flength}, every finitely generated $\CC$-module has a finite support. Hence every locally representation-finite $\k$-category is locally support-finite. We should mention that there are locally representation-infinite $\k$-categories which are locally support-finite; the reader may consult \cite{DS}.\\

\begin{sproposition}\cite[Theorem]{DLS}\label{loc-supp-fin} Let $\CC$ be a locally support-finite category, and let $G$ act freely on $\ind \CC$. Then $\CC/G$ is also locally support-finite, and the push-down functor $P_{*}:\mmod \CC \to \mmod (\CC G)$ is dense.
\end{sproposition}
\s \label{lifting} Let $\CC$ be a $G$-category with a free $G$-action on $\ind \CC$. Based on Theorem \ref{Corres}, the functor $P_{*}: \mmod \CC \to \mmod(\CC/G)$ is a Galois $G$-precovering. Consequently, for any two finitely generated $\CC$-modules $X$ and $Y$, $P_{*}$ induces the following $\k$-isomorphism
\[\bigoplus_{a \in G} \Hom_{\mathcal{C}}(X,{}^a Y) \rightarrow \Hom_{\mathcal{C}/G}(P_{*}(X), P_{*}(Y)).\]

Thus, every morphism $\theta: P_{*}(X) \to P_{*}(Y)$ can be expressed as $\theta = \sum_{a \in G} P_{*}(f_a)$, where $f_a: X \to {}^a Y$, and only finitely many of these morphisms $f_a$ are non-zero. In this context, we say that $(f_a: X \to {}^a Y)_a$ is a lifting of $\theta$ along $P_{*}$.

\section{Induced Galois $G$-(pre)covering on subcategories}
Throughout this paper, $\CC$ denotes a locally bounded $G$-category. A full subcategory $\CU$ of $\mmod \CC$ is called $G$-equivariant if ${}^a \CU = \CU$ for all $a \in G$. Let us recall the notion of $n$ -Auslander-Reiten translation on $\mmod \CC$.

Let $D:= \Hom_{\k}(-, {\k}) \colon \Mod {\k} \to \Mod {\k}$ be the usual ${\k}$-dual functor. Then it induces the following contravariant functors
\[D:= D \circ (-):\Mod \CC \lrt \Mod \CC^{\rm{op}}\ \text{and} \ D:= D \circ (-):\Mod \CC^{\rm{op}} \lrt \Mod \CC \]
defined by $M \mapsto DM:= D \circ M$, which are also denoted by $D$.
It follows from \cite[Proposition 2.6]{AR3} that the above functors can be restricted to the dualities
\[D:\mmod \CC\lrt \mmod \CC^{\rm{op}}\ \text{and} \ D:\mmod \CC^{\rm{op}}\lrt \mmod \CC,\]
respectively.

Let $\Omega: \underline{\rm mod}\mbox{-}\CC \lrt \underline{\rm mod}\mbox{-}\CC$, resp. $\Omega^-: \overline{\rm mod}\mbox{-}\CC \lrt \overline{\rm mod}\mbox{-}\CC$, denote the syzygy functor, resp. cosyzygy functor, where $\underline{\rm mod}\mbox{-}\CC$, resp. $\overline{\rm mod}\mbox{-}\CC$, is the stable category of $\mmod\CC$ modulo projective objects, resp. the costable category of $\mmod\CC$ modulo injective objects.

For $n\geq 1$, $n$-Auslander-Reiten translations are defined in \cite{I} as follows
\begin{align*}
\tau_n:= \tau \Omega^{n-1}: \underline{\rm mod}\mbox{-}\CC \lrt \overline{\rm mod}\mbox{-}\CC, \\
\tau^-_n:= \tau^-\Omega^{-(n+1)} : \overline{\rm mod}\mbox{-}\CC \lrt \underline{\rm mod}\mbox{-}\CC.
\end{align*}
To compute, for every functor $X$ of $\mmod \CC$ take a projective resolution and an injective resolution
\[ P_n \st{f} \lrt P_{n-1} \lrt \cdots \lrt P_0 \lrt X \lrt 0 \]
and
\[0 \lrt X \lrt I^0 \lrt \cdots \lrt I^{n-1} \st{g} \lrt I^n\]
respectively. Then
\[ \tau_n X = \tau(\Ker f) \ \ \text{and} \ \ \tau_n^- X = \tau^-(\Coker g).\]

\subsection{Induced Galois $ G$-(pre)covering on $ \CP_n$ and $ \CI_n$}
Let $n \geq1$. A subcategory $\CC$ of $\mmod \La$ is called $n$-cluster tilting if $\CU$ is functorially finite and $\CU = {}^{\perp_{n-1}}\CU = \CU^{\perp_{n-1}}.$
In \cite[Definition 3.2]{IS} the authors introduced the notion of an $n$-precluster tilting subcategory as a generalization of an $n$-cluster tilting subcategory.
A subcategory $\CU$ of $\mmod \La$ is called an $n$-precluster tilting subcategory if it is a functorially finite generator-cogenerator subcategory of $\mmod \La$, and satisfies the following conditions.
\begin{itemize}
 \item [$(i)$] $\tau_n(\CU) \subseteq \CU$ \ and \ $\tau_n^-(\CU) \subseteq \CU$,
 \item [$(ii)$] $\Ext^i_{\La}(\CU, \CU) =0$, for $0 <i <n$.
\end{itemize}
If $\CC$ admits an additive generator $M$, we say that $M$ is an $n$-precluster tilting module.\\

The following definition is motivated by the notion of an $n$-precluster tilting subcategory \cite{IS}.

\begin{sdefinition}\label{n-precluster}
An additively closed $G$-equivariant subcategory $\CU$ of $\mmod \CC$ is called $n$-precluster tilting subcategory if
	\begin{itemize}
		\item[$(i)$] $\CU$ is a generator-cogenerator subcategory of $\mmod \CC$ (i.e. for every object $x$ of $\CC$, ${\CC}(-, x)$ and $D{\CC}(x, -)$ belong to $\CU$);
		\item[$(ii)$] $\tau_n(\CU) \subseteq \CU$ and $\tau_n^-(\CU)\subseteq \CU$;
		\item [$(iii)$] $\Ext^i_{\CC}(\CU, \CU)=0$ for all $0<i<n$;
		\item[$(iv)$] $\CU$ is functorially finite.
	\end{itemize}
Moreover, if there is a $\CC$-module $M$ such that $\CU= \add M$, resp. $\CU= \add \{{}^a M \mid a \in G\}$, then $M$ is called an $ n$-preculster tilting $\CC$-module, resp. $(G, n)$-preculster tilting $\CC$-module.
\end{sdefinition}
$\tau_n$-selfinjective algebras were introduce in \cite{IS} as an algebra with an $n$-precluster tilting module.

\begin{sdefinition}
A locally bounded $G$-catgeory $\CC$ is called $\tau_n$-selfinjective if $ \CC$ admits an $ n$-preculster tilting module.
\end{sdefinition}

We have the following two subcategories of $\mmod \CC$

\[ \CP_n(\CC)= \add \{\tau_n^{-i} ({\CC}(-,x))\mid i\geq 0 , \ x \in \CC\}\]
\[ \CI_n(\CC)= \add \{ \tau_n^i(D {\CC}(x, -))\mid i\geq 0, \ x \in \CC\}\]

Moreover, for a subcategory $\CX$ of $\mmod \CC$ , the following full subcategory can be defined:
\[\CX^{\perp_{n-1}}=\{ M \in \mmod \CC \mid \Ext_{\CC}^i(\CX, M)=0 \ \ \text{for all} \ \ 0<i<n\},\]
\[{}^{\perp_{n-1}} \CX=\{ M \in \mmod \CC \mid \Ext_{\CC}^i(M, \CX)=0 \ \ \text{for all} \ \ 0<i<n\}.\]

Now we are ready to prove the first main result of this section.

\begin{stheorem}\label{P_n(C)}
Let $\CC$ be a locally bounded $G$-category and $G$ acts freely on $\ind \CC$. Then
\begin{itemize}
	\item[$(i)$] the push-down functor can be restricted to the following functors
	\[P_*: \CP_n(\CC) \lrt \CP_n(\CC/G), \ \text{and} \ P_*: \CI_n (\CC) \lrt \CI_n(\CC/G)\]
	that are dense;
	\item[$(ii)$] the push-down functor induces a Galois $G$-covering
	\[P_*: \ind (\CP_n(\CC))\lrt \ind (\CP_n(\CC/G)).\]
	In particular, $\ind (\CP_n(\CC/G))\simeq (\ind (\CP_n(\CC)))/G$.
\end{itemize}

\end{stheorem}

\begin{proof}
$(i)$ It is enough to show that 	$P_{*}(\CP_n(\CC))\simeq\CP_n(\CC/G)$ and
$P_{*}(\CI_n(\CC))\simeq\CI_n(\CC/G)$. To this end, consider a $\CC$-module $M$ together with a projective resolution
\[ \cdots \lrt P_2 \st{g_2}\lrt P_1 \st{g_1} \lrt P_0 \st{g_0} \lrt M \lrt 0.\]

Since $P_{*}$ is an exact functor which takes projective $\CC$-modules to projective $\CC/G$-modules, we have the following projective resolution of $P_{*}(M)$
\[ \cdots \lrt P_{*}(P_2) \st{P_{*}(g_2)} \lrt P_{*}(P_1) \st{P_{*}(g_1)} \lrt P_{*}(P_0) \st{P_{*}(g_0)} \lrt P_{*}(M) \lrt 0.\]

In view of [3.2]\cite{BG}, the push-down functor $P_{*}$ preserves the transpose functor. Therefore, $P_{*}(\tau_n M)\cong \tau_n P_{*}(M)$ and $P_{*}(\tau^-_n M)\cong \tau^-_n P_{*}(M)$.
		
It follows directly from the definition that, for each $x \in \CC$, $P_{*}\CC(-,x)\cong \CC/G(-, Px)$. This fact together with the above argument implies that
\[P_{*}(\CP_n(\CC)) \subseteq \CP_n(\CC/G) \ \text{and} \ P_{*}(\CI_n(\CC))\subseteq \CI_n(\CC/G).\]

For the converse inclusion observe that any indecomposable projective $\CC/G$-module is of the form $\CC/G(-, Px)$, because $P: \CC \lrt \CC/G$ is a Galois $G$-covering functor.
		
 Let $X$ be an object of $\CP_n(\CC/G)$. First we assume that $X$ is of the following form
\[ \oplus_{j=1}^n \tau_n^{-i_j} (\CC/G(-, Px_j)).\]
We have the following isomorphisms
		\[ \begin{array}{llll}
		 	X & = \oplus_{j=1}^n \tau_n^{-i_j} (\CC/G(-, Px_j))\\
		 	 	& \cong \oplus_{j=1}^n \tau_n^{-i_j} (P_*(\CC(-, x_j))) \\
		 	 	& \cong \oplus_{j=1}^n P_*( \tau_n^{-i_j} (\CC(-, x_j)))\\
		 	 	& \cong P_*( \oplus_{j=1}^n \tau_n^{-i_j} (\CC(-, x_j))).
		 \end{array}\]
Observe that the last two isomorphisms follows from the fact that the push-down functor $P_*$ commutes with $\tau_n^-$ and direct sums. That is $X$ lies in $P_*(\CP_n(\CC))$. Since the push-down functor $P_*$ preserves direct summand, this proof shows that the functor $P_*: \CP_n(\CC) \lrt \CP_n(\CC/G)$ is dense.

$(ii)$ In view of part $(i)$, we can take
$\ind (\CP_n(\CC/G))=\{P_*(X) \mid X \in \ind (\CP_n(\CC))\}.$ Since $P_*: \CP_n(\CC) \lrt \CP_n(\CC/G)$ is $G$-precovering, by \cite[Lemma 2.6]{BL}, $P_*$ sends decomposable objects to decomposable objects. Hence, by Theorem \ref{Corres} $(i)$, every object $X$ in $\CP_n(\CC)$ is indecomposable if and only if $P_*(X)$ is indecomposable in $\CP_n(\CC/G)$. Now the result follows from part $(i)$.
\end{proof}

\subsection{Induced Galois $G$-(pre)covering on $n$-precluster tilting subcategories}
In this subsection, it is proved that the push-down functor $P_*: \mmod \CC \lrt \mmod (\CC/G)$ preserves $n$-precluster tilting subcategories. Using this result, we investigate the relation between $\tau_n$-selfinjectiveness of $\CC$ and $\CC/G$.

\begin{slemma}\label{approximation}
 Let $\CC$ be a locally support-finite $G$-category with a free $G$-action on $\ind \CC$. Assume that $\CU$ is a functorially finite subcategory of $\mmod \CC$ which is $G$-equivariant and is closed under finite direct sums. Then $P_*(\CU)$ is a functorially finite subcategory of $\mmod(\CC/G)$.
\end{slemma}

\begin{proof}
We only prove that $P_*(\CU)$ is contravariantly finite, the covariantly finite case follows similarly. Consider an object $Y$ of $\mmod(\CC/G)$. Since $P_*$ is dense, by Proposition \ref{loc-supp-fin}, $Y \cong P_*(X)$ for some $X \in \mmod \CC$. By assumption there is an object $M$ in $\CU$ together with a right $\CU$-approximation $f: M \lrt X$ in $\mmod \CC$. We claim that
$P_{*}(f): P_{*}(M) \lrt P_{*}(X)$ is a right $P_{*}(\CU)$-approximation in $\mmod(\CC/G)$ .

Let there is a morphism $g: P_{*}(M') \lrt P_{*}(X) $ with $M' \in \CU$. By \ref{lifting}, $g: P_{*}(M') \lrt P_{*}(X) $ can be written as
\[g= \sum_{a \in G} P_{*}({u}_a),\]
where ${u}_a: {}^a M' \lrt X$ and only finitely many of these maps ${u}_a$ being non-zero. Hence, there is a morphism $u: \oplus_{a \in G'} {}^a M' \lrt X$ where
$u=(u_a)_{a \in G'}$ and $G'$ is a finite subgroup of $G$ such that $a \in G'$ if and only if $u_a$ is non-zero. Since $\CU$ is $G$-equivariant and is closed under finite direct sums, $\oplus_{a \in G'} {}^a M'$ belongs to $\CU$. Therefore, for each $a \in G'$, there is a morphism $v_a: {}^a M' \lrt M$ making the following diagram commutative
\[\xymatrix{ M \ar[r]^f & X \\ {}^a M'. \ar[u]^{v_a} \ar[ru]_{u} & }\]
Take $h= \sum_{a \in G'} P_{*}({v}_a)$. Then we have a morphism $h: P_{*}(M') \lrt P_{*}(M)$ such that
\[ \begin{array}{lllll}
 P_{*}(f) \circ h & = P_{*}(f) \circ (\sum_{a \in G'} P_{*}({v}_a)) \\
 & = \sum_{a \in G'} (P_{*}(f) \circ P_{*}({v}_a)) \\
 & = \sum_{a \in G'} P_{*}(f \circ {v}_a)\\
 & = \sum_{a \in G'} P_{*}({u}_a)\\
 & = g.
 \end{array}\]
The proof now is complete.
\end{proof}

We need the following lemma.

\begin{slemma}\cite[Lemma 3.5]{DI}\label{DILemma}
Let $\CC$ be a locally bounded $G$-category with free $G$-action on $\ind \CC$. Then there is a functorial isomorphism
\[\Ext_{\CC/G}^i (P_{*}(X), P_{*}(Y)) \cong \oplus_{a \in G}\Ext_{\CC}^i(X, {}^a Y),\]
for every $i \geq 0$ and $X, Y \in \mmod \CC$.
\end{slemma}

\begin{stheorem}\label{Main1}
Let $\CC$ be a locally bounded $G$-category with free $G$-action on $\ind \CC$.
	\begin{itemize}
		\item[$(i)$] If $\CU$ is a $G$-equivariant $n$-precluster tilting subcategory of $\mmod \CC$ such that $P_{*}(\CU)$ is functorially finite in $\mmod(\CC/G)$, then $P_{*}(\CU)$ is an $n$-precluster tilting subcategory of $\mmod(\CC/G)$.
		\item[$(ii)$] If $\CC$ is, in addition, locally support-finite and $\CU$ is a $G$-equivariant $n$-precluster tilting subcategory of $\mmod \CC$, then $P_{*}(\CU)$ is an $n$-precluster tilting subcategory of $\mmod(\CC/G)$.
	\end{itemize}
\end{stheorem}

\begin{proof}
$(i)$ First we check that $P_{*}(\CU)$ is a generator-cogenerator subcategory of $\mmod(\CC/G)$. Since the functor $P: \CC \lrt \CC/G$ is surjective on objects, it is enough to show that $P_{*}(\CU)$ contains $\CC/G(-, Px)$ and $D(\CC/G(Px,-))$ for each $x \in \CC$. By definition of the push-down functor, $P_{*}(\CC(-, x))\cong \CC/G(-, Px)$ and $P_{*}(D\CC(x,-))\cong D(\CC/G(Px,-))$. Thus $P_*(\CU)$ contains $\CC/G(-, Px)$ and $D(\CC/G(Px, -))$ for every object $x$ of $\CC$. This means that $P_*(\CU)$ is a generator-cogenerator subcategory of $\mmod(\CC/G)$.
	
In view of the proof of Proposition \ref{P_n(C)}, there exist natural isomorphisms $P_{*} \tau_n \cong \tau_n P_{*}$ and $P_{*}\tau^-_n \cong \tau^-_n P_{*}$. Hence
\[\tau_n(P_{*}(\CC)) \subseteq P_{*}(\CC) \ \text{and} \ \tau^-_n(P_{*}(\CC)) \subseteq P_{*}(\CC).\]
Moreover, by Lemma \ref{DILemma}, for every $Y_1 , Y_2 \in \CU$ and every $0<i<n$, there is a functorial isomorphism
\[\Ext_{\CC/G}^i (P_{*}(Y_1), P_{*}(Y_2)) \cong \oplus_{a \in G}\Ext_{\CC}^i(Y_1, {}^a Y_2).\]
Since $\CU$ is $n$-precluster tilting and $G$-equivariant, $\Ext_{\CC}^i(Y_1, {}^a Y_2)=0$ for all $a \in G$. This means that $\Ext_{\CC/G}^i (P_{*}(Y_1), P_{*}(Y_2))=0$ as well. The proof now is complete.
	
$(ii)$ This follows from Lemma \ref{approximation} and part$(i)$.
\end{proof}

\begin{stheorem}\label{Main2}
Let $\CC$ be a locally bounded $G$-category with free $G$-action on $\ind \CC$. If $\CV \subseteq P_{*}(\mmod \CC) $ is an $n$-precluster tilting subcategory of $\mmod(\CC/G)$, then
\[P_{*}^{-1}(\CV)=\{X \in \mmod \CC| P_{*}(X)\in \CV\}\]
is a $G$-equivariant $n$-precluster tilting subcategory of $\mmod \CC$.
\end{stheorem}
\begin{proof}
Let $\CU = P_{*}^{-1}(\CV)$. Since $P_{*}(X) \cong P_{*}({}^a X)$ for each $a \in G$, $\CU$ is $G$-equivariant.
Consider an object $x$ of $\CC$. Since $\CV$ is a generator-cogenerator subcategory of $\mmod(\CC/G)$, it contains $\CC/G$-modules $\CC/G(-, Px)$ and $D(\CC/G(Px,-))$. By definition, there is a $\CC$-module $X$ such that $P_{*}(X)\cong \CC/G(-, Px)\cong P_{*}(\CC(-,x))$. Furthermore, by applying the pull-up functor $P^*$ on the isomorphism $P_*(X) \cong P_*(\CC(-,x))$, in view of \ref{push-down}, we have the following canonical isomorphism
\[ \oplus _{a \in G} {}^a X \cong \oplus_{a \in G} {}^a \CC(-, x).\]
 Hence $\CC(-,x)$ is a direct summand of $\oplus _{a \in G} {}^a X$ and so there is an object $a \in G$ and a direct summand $Y$ of $X$ such that $\CC(-,x)= {}^a Y$, because $\CC(-, x)$ is an indecomposable $\CC$-module. This means that $\CC(-,x)$ belongs to $\CU$. The same argument can be applied to see that $D\CC(x,-)$ belongs to $\CU$. Therefore, $\CU$ is a generator-cogenerator subcategory of $\mmod \CC$.

Now, let $X$ be an object of $\CU$. By the proof of Proposition \ref{P_n(C)}, $P_{*} \tau_n (X) \cong \tau_n P_{*}(X)$. Since $\CV$ is an $n$-precluster tilting subcategory, $\tau_n P_{*}(X) \in \tau_n( \CV) \subseteq \CV$. This, in turn, implies that $\tau_n (X) \in \CU$ and so $\tau_n(\CU) \subseteq \CU$. A similar argument works to show that $\tau_n^-(\CU) \subseteq \CU$.

For any two objects $Y_1 $ and $Y_2$ of $\CU$ and $0<i<n$, $\Ext_{\CC}^i(Y_1, Y_2)$ is a direct summand of $\oplus_{a \in G}\Ext_{\CC}^i(Y_1, {}^a Y_2)$ and, by Lemma \ref{DILemma},
$$\Ext_{\CC/G}^i (P_{*}(Y_1), P_{*}(Y_2)) \cong \oplus_{a \in G}\Ext_{\CC}^i(Y_1, {}^a Y_2).$$
By definition of $n$-precluster tilting subcategories, $\Ext_{\CC/G}^i (P_{*}(Y_1), P_{*}(Y_2))=0$, and so
\[\oplus_{a \in G}\Ext_{\CC}^i(Y_1, {}^a Y_2)=0\]
for all $0<i<n$. Thus, $\Ext_{\CC}^i(Y_1, Y_2)=0$ for all $0<i<n$, as well.

Finally, it follows from Lemma 3.7 of \cite{DI} that $\CU$ is functorially finite. This completes the proof.
\end{proof}

\begin{scorollary}\label{tau_nSelfinj}
Let $\CC$ be a locally support-finite $G$-category with free $G$-action on $\ind \CC$. Then $\CC$ admits a $(G, n)$-precluster tilting module if and only if $\CC/G$ is $\tau_n$-selfinjective.
\end{scorollary}

\begin{proof}
Let $M\in \mmod \CC$ be a $(G, n)$-precluster tilting module. Then $\CU= \add \{{}^a M \mid a \in G\}$ is an $n$-precluster tilting subcategory of $\mmod \CC$. So $P_{*}(\CU)$ is an $n$-precluster tilting subcategory of $\mmod(\CC/G)$, by Theorem \ref{Main1} $(i)$. Moreover, it can be easily checked that $P_{*}(\CU)\cong \add P_{*}(M)$. That is $\CC/G$ is $\tau_n$-selfinjective.

For the converse, suppose that $Y$ is an $n$-precluster tilting $\CC/G$-module. By Proposition \ref{loc-supp-fin}, the push-down functor is dense and so $Y \cong P_{*}(X)$ for some $\CC$-module $X$. Take $\CU:=P_{*}^{-1}(\add Y)$ which is $n$-precluster tilting by Theorem \ref{Main2}. It is enough to show that $\CU= \add \{{}^a X \mid a \in G\}$. Clearly $\add \{{}^a X \mid a \in G\} \subseteq \CU$. Now, assume that a $\CC$-module $M$ belongs to $\CU$. Then $P_{*}(M) \in \add Y$. Since $P_{*}$ preserves direct summands, we may assume that $P_{*}(M) = \oplus_{I} P_{*}(X)$, where $I$ is a finite set. By \ref{push-down}, we have the following canonical isomorphism
\[ \oplus_{a \in G} {}^a M\cong \oplus_{a \in G} {}^a (\oplus_{I} X). \]
Hence, $M$ is a direct summand of $\oplus_{a \in G} {}^a (\oplus_{I} X)$ and so there is a finite subset $H$ of $G$ such that
\[M \cong \oplus_{a \in H} {}^a (\oplus_{I } X),\]
thanks to \cite[Theorem 1]{W}. This means that $M$ belongs to $\add \{{}^a X \mid a \in G\}$. Therefore, $\CU \subseteq \add \{{}^a X \mid a \in G\}$. This completes the proof.
\end{proof}

The same argument as in the Proposition 3.5 of \cite{IS} works to prove the following propositions. So we skip their proofs.

\begin{sproposition}\label{Prop1}
	Let $\CC$ be a locally bounded $k$-category. Let $\CI_n:= \CI_n(\CC)$ and $\CP_n:=\CP_n(\CC)$. Then the following conditions are equivalent.
	\begin{itemize}
		\item[$(i)$] $\CC$ is $\tau_n$ selfinjective.
		\item[$(ii)$] $\CI_n$ is of finite type, $\CI_n \subseteq {}^{\perp_{n-1}} \{\CC(-, x) \mid x \in \CC\}$ and $\Ext^i_{\CC} (\CI_n, \CI_n) =0$ for $0 <i <n$.
		\item[$(iii)$] $\{\CC(-,x) \mid x \in \CC\}\subseteq \CI_n$ and $\Ext^i_{\CC} (\CI_n, \CI_n) =0$ for $0 <i <n$.
		\item[$(iv)$] $\CP_n$ is of finite type, $\CP_n \subseteq \{D\CC(x,- )\mid x \in \CC\}^{\perp_{n-1}}$ and $\Ext^i_{\CC}(\CP_n, \CP_n) =0$ for $0 <i <n$.
		\item[$(v)$] $\{ D \CC(x,-) \mid x \in \CC\} \subseteq \CP_n$ and $\Ext^i_{\CC}(\CP_n, \CP_n) =0$ for $0 <i <n$.
	\end{itemize}
\end{sproposition}

\begin{sproposition}\label{Prop2}
	Let $\CC$ be a locally bounded $G$-category with free $G$-action on $\ind \CC$. Let $\CI_n:= \CI_n(\CC)$ and $\CP_n:=\CP_n(\CC)$. Then the following conditions are equivalent.
	\begin{itemize}
		\item[$(i)$] $\CC$ admits a $(G, n)$-preculster tilting module.
		\item[$(ii)$] $\CI_n$ is of locally finite type, $\CI_n \subseteq {}^{\perp_{n-1}} \{\CC(-, x) \mid x \in \CC\}$ and $\Ext^i_{\CC} (\CI_n, \CI_n) =0$ for $0 <i <n$.
		\item[$(iii)$] $\{\CC(-,x) \mid x \in \CC\}\subseteq \CI_n$ and $\Ext^i_{\CC} (\CI_n, \CI_n) =0$ for $0 <i <n$.
		\item[$(iv)$] $\CP_n$ is of locally
		finite type, $\CP_n \subseteq \{D\CC(x,- )\mid x \in \CC\}^{\perp_{n-1}}$ and $\Ext^i_{\CC}(\CP_n, \CP_n) =0$ for $0 <i <n$.
		\item[$(v)$] $\{ D \CC(x,-) \mid x \in \CC\} \subseteq \CP_n$ and $\Ext^i_{\CC}(\CP_n, \CP_n) =0$ for $0 <i <n$.
	\end{itemize}
\end{sproposition}

\begin{scorollary}
Let $\CC$ be a locally support-finite $G$-category with free $G$-action on $\ind \CC$. Then the following conditions are equivalent.
\begin{itemize}
		\item[$(i)$] $\CI_n(\CC)$ is of locally finite type, $\CI_n(\CC) \subseteq {}^{\perp_{n-1}} \{\CC(-, x) \mid x \in \CC\}$ and $\Ext^i_{\CC} (\CI_n(\CC), \CI_n(\CC)) =0$ for $0 <i <n$.
		\item[$(ii)$] $\CI_n(\CC/G)$ is of finite type, $\CI_n(\CC/G) \subseteq {}^{\perp_{n-1}} \{\CC/G(-, u) \mid u \in \CC/G\}$ and \\
		$\Ext^i_{\CC/G} (\CI_n(\CC/G), \CI_n(\CC/G)) =0$ for $0 <i <n$.
		\item[$(iii)$] $\CP_n(\CC)$ is of locally
		finite type, $\CP_n(\CC) \subseteq \{D\CC(x,- )\mid x \in \CC\}^{\perp_{n-1}}$ and $\Ext^i_{\CC}(\CP_n(\CC), \CP_n(\CC)) =0$ for $0 <i <n$.
		\item[$(iv)$] $\CP_n(\CC/G)$ is of finite type, $\CP_n(\CC/G) \subseteq \{D(\CC/G(u,- ))\mid u \in \CC/G\}^{\perp_{n-1}}$ and \\ $\Ext^i_{\CC/G}(\CP_n(\CC/G), \CP_n(\CC/G)) =0$ for $0 <i <n$.
\end{itemize}
\end{scorollary}

\begin{proof}
Corollary \ref{tau_nSelfinj} together with Propositions \ref{Prop1} and \ref{Prop2} implies the result
\end{proof}

\begin{sremark}\label{Remark3.12}
Let $\CU$ be a functorially finite subcategory of $\mmod \CC$ that satisfies conditions $(i)$ and $(iii)$. The same argument as in the proof of Proposition 3.8 of \cite{IS} works to see that $\CU$ is an $n$-precluster tilting subcategory if and only if ${}^{\perp_{n-1}} \CU= \CU^{\perp_{n-1}}$. Set
\[\CZ(\CU):={}^{\perp_{n-1}} \CU= \CU^{\perp_{n-1}}.\]
\end{sremark}

\begin{sproposition}\label{Z(U)}
Let $\CC$ be a locally bounded $G$-category with free $G$-action on $\ind \CC$.
\begin{itemize}
\item[$(i)$] If $\CU$ is a $G$-equivariant $n$-precluster tilting subcategory of $\mmod \CC$ such that $P_{*}(\CU)$ is functorially finite in $\mmod(\CC/G)$, then $P_*$ induces a $G$-precovering $P_*: \CZ(\CU) \lrt \CZ(\CU/G)$.
\item[$(ii)$] If $\CC$ is, in addition, locally support-finite and $\CU$ is a $G$-equivariant $n$-precluster tilting subcategory of $\mmod \CC$, then $P_*: \CZ(\CU) \lrt \CZ(\CU/G)$ is dense.
\end{itemize}
\end{sproposition}

\begin{proof}
$(i)$	In view of \cite[Lemma 3.5]{DI}, $P_*(\CU)\simeq \CU/G$. By Theorem \ref{Main1}, $\CU/G$ is an $n$-precluster tilting subcategory of $\mmod(\CC/G)$. Hence we have $\CZ(\CU/G)={}^{\perp_{n-1}} \CU/G= \CU/G^{\perp_{n-1}}$. Let $X$ be an object of $\CZ(\CU)$. According to Lemma \ref{DILemma}, for every $X' \in \CU$
\[\Ext^i_{\CC/G}(P_*(X), P_*(X')) \cong \oplus_{a \in G} \Ext^i_{\CC}(X, {}^a X').\]
This implies that for every $0<i<n$, $\Ext^i_{\CC/G}(P_*(X), P_*(X'))=0$. That is $P_*(X)$ lies in $\CZ(\CU/G)$.
	
$(ii)$ Consider an object $Y$ of $\CZ(\CU/G)$. Since $P_*: \mmod \CC \lrt \mmod(\CC/G)$ is dense, by Proposition \ref{loc-supp-fin}, there is a $\CC$-module $X$ in $\mmod \CC$ such that $P_*(X)\cong Y$. Let $U$ be an object of $\CU$. There is an isomorphism
\[\Ext^i_{\CC/G}(P_*(X), P_*(U)) \cong \oplus_{a \in G} \Ext^i_{\CC}(X, {}^a U)\]
by Lemma \ref{DILemma}. By assumption $\Ext^i_{\CC/G}(P_*(X), P_*(U)) =0$ for all $0<i<n$. This in turn implies that $\oplus_{a \in G} \Ext^i_{\CC}(X, U)$ and so its direct summand $ \Ext^i_{\CC}(X, {}^a U)$ vanishes for every $0<i<n$.
\end{proof}

Toward the end of this subsection, we aim to characterize $\CZ(\CU)$ in terms of Gorenstein projective functors. We do this by applying relative homological algebra techniques introduced by Auslander and Solberg \cite{AS1}. Let us start with some background.

For a subcategory $\CU$ of $\mmod \CC$, let $F_{\CU}$ be an additive sub-bifunctor of $\Ext_{\CC}^i(-,-)$, that is, for each pair of $\CC$-modules $X$ and $Y$, $F_{\CU}(Y,X)$ is consisting of all short exact sequences
\[0 \lrt X \lrt Z \lrt Y \lrt 0\]
that remains exact after applying the Hom-functor $\Hom_{\CC}(U, -)$, for all $U \in \CU$. These kind of short exact sequences are called $F_{\CU}$-exact.

A $\CC$-module $P$ in $\mmod \CC$ is called $F_{\CU}$-projective, if all $F_{\CU}$-exact sequence
\[0 \lrt X \lrt Y \lrt P \lrt 0\]
splits. The full subcategory of $\mmod \CC$ consisting of all $F_{\CU}$-projective $\CC$-modules will be denoted by $\CP(F_{\CU})$. In a duall manner, one can define $F_{\CU}$-injective functors in $\mmod \CC$. Let $\CI(F_{\CU})$ denote the full subcategory of $\mmod \CC$ consisting of all $F_{\CU}$-injectives.
It follows from \cite[Proposition 1.10]{AS1} that
\[\CP(F_{\CU})= \add \{ \CU, \prj \CC\} \ \ \text{and} \ \ \CI(F_{\CU})= \add \{ \tau \CU, \inj \CC\}.\]
Assume that $\mmod \CC$ has enough $F_{\CU}$-projectives and $F_{\CU}$-injectives, that is, for every $X \in \mmod \CC$ there are $F_{\CU}$-exact sequences
\[ 0\lrt K \lrt P \lrt X \lrt 0 \]
and
\[ 0 \lrt X \lrt I \lrt C \lrt 0 \]
in $\mmod \CC$ with $P \in \CP(F_{\CU})$ and $I \in \CI(F_{\CU})$. Hence, a $F_{\CU}$-projective resolution of $X$ is a $F_{\CU}$-exact sequence
\[ \cdots \lrt P _2 \lrt P_1 \lrt P_0 \lrt X \lrt 0 \]
with $P_i \in \CP(F_{\CU})$. Similarly, we have a $F_{\CU}$-injective resolution for every object of $\mmod \CC$.

Consider a pair of finitely generated $\CC$-modules $X$ and $Y$, along with a $F_{\CU}$-projective resolution ${\bf P}_X$ for $X$ and a $F_{\CU}$-injective resolution ${\bf I}_Y$ for $Y$. It has been shown in \cite[Section 2]{AS1} that the $i$-th homology group of the sequence $\Hom_{\CC}({\bf P}_Y, X)$ is isomorphic to the $i$-th homology group of the sequence $\Hom_{\CC}(Y, {\bf I}_X)$ for every $i \geq 1$. Consequently, we can define the $i$-th $F_{\CU}$-relative extension group $\Ext_{F_{\CU}}^i(Y, X)$ as the $i$-th homology group of the sequence $\Hom_{\CC}({\bf P}_Y, X)$ or equivalently, as the $i$-th homology group of the sequence $\Hom_{\CC}(Y, {\bf I}_X)$. For further details, the reader may refer to \cite[Section 2]{AS1} or \cite[Section 2.1]{IS}.

\begin{slemma}
Let $\CU$ be an $n$-precluster tilting subcategory of $\mmod \CC$ and $F=F_{\CU}$. Then $${}^{\perp_F} \CU = \CZ(\CU),$$ where
${}^{\perp_F} \CU = \{X \in \mmod \CC \mid \Ext_F^i(X, \CU)=0 \ \ \text{for all} \ \ i>0\}$.
\end{slemma}

\begin{proof}
By definition, we have
\[\Ext^i_{\CC}(\CU, \CU) = 0, \quad \text{for all } 0 < i < n.\]
Thus, by \cite[Proposition 1.3]{L}, for each $\CC$-module $X$ in $\mmod\CC$, we obtain
\[\Ext^i_{F}(X, \CU) = \Ext_{\CC}^i(X, \CU), \quad \text{for all } 0 < i < n.\]

This immediately implies the inclusion
\[{}^{\perp_F} \CU \subseteq \CZ(\CU).\]

Conversely, assume that $X \in \CZ(\CU)$. By the previous remark, we have
\[\Ext^i_{F}(X, \CU) = 0, \quad \text{for all } 0 < i < n.\]
Next, consider an object $U \in \CU$ along with an injective resolution
\[0 \longrightarrow U \longrightarrow I^0 \longrightarrow I^1 \longrightarrow \cdots \longrightarrow I^{n-2} \longrightarrow \Omega^{-(n-1)}(U) \longrightarrow 0,\]
which is $F$-exact due to the condition
\[\Ext^i_{\CC}(\CU, \CU) = 0, \quad \text{for all } 0 < i < n.\]

Since $\tau_n^-(\CU) \subseteq \CU$ and
\[\CI(F) = \add\{\tau(\CU), D\CC(-,x) \mid x \in \CC\},\]
it follows that $\Omega^{-(n-1)}(X)$ belongs to $\mathcal{I}(F)$. Consequently, we obtain
\[\Ext^i_{F}(X, \CU) = 0, \quad \text{for all } i \geq n.\]

Thus, $X \in {}^{\perp_F} \CU$, and we conclude that
\[\CZ(\CU) \subseteq {}^{\perp_F} \CU.\]
\end{proof}

Let $\CU$ be a subcategory of $\mmod \CC$. A totally acyclic complex of projective $\CU$-modules is an exact sequence
\[{\bf P}: \cdots \lrt P_{i+1} \lrt P_{i} \lrt P_{i-1} \lrt \cdots\]
of projective functors in $\mmod \CU$, such that the induced complex $\Hom({\bf P}, \CU(-, U))$ is exact for every object $U$ in $\CU$. Since $\CU$ is an additive subcategory, every finitely generated projective $\CU$-module can be expressed as $\CU(-, U) = \Hom_{\CC}(-, U)|_{\CU}$ for some $U \in \CU$.

A finitely presented $\CU$-module $X$ is called Gorenstein projective if there exists a totally acyclic complex
\[{\bf P}: \cdots \lrt P_{i+1} \lrt P_{i} \lrt P_{i-1} \lrt \cdots\]
of projectives such that $X = \Ker(P_i \lrt P_{i-1})$ for some $i \in \Z$. We denote the full subcategory of $\mmod \CU$ consisting of all Gorenstein projective objects by $\Gprj \CU$. Note that if $\La$ is a Gorenstein Artin algebra, then $\Gprj \La = {\rm CM}\mbox{-} \La$, as referenced in \cite{IS}.

\begin{sproposition}\label{equivalence}
Let $\CC$ be a locally bounded $k$-category and $\CU$ be an $n$-precluster tilting subcategory of $\mmod \CC$. Then there is an equivalence between $\CZ(\CU)$ and $\Gprj \CU$.
\end{sproposition}

\begin{proof}
First we prove that there is an equivalence $\Phi: {}^{\perp_F}\CU \lrt \Gprj \CU$ such that for each object $X$ of ${}^{\perp_F} \CU$, $\Phi(X)= \Hom_{\CC}(-,X)|_{\CU}$. By Yoneda Lemma, for every functor $U$ of $\CU$ and every $\CC$-module $X$,
\[ \Hom (\CU(-,U) , \Hom_{\CC}(-,X)|_{\CU}) \cong \Hom_{\CC}(U,X).\]
Consider an object $X$ of $ {}^{\perp_F} \CU$. Since $\CU$ is functorially finite, there is an $\Hom_{\CC}( \CU,-)$-exact sequence
\[\cdots \lrt U_2 \lrt U_1 \lrt U_0 \lrt X \lrt 0 \]
such that for every $i\geq 0$, $U_i \in \CU$. Then we get an exact sequence
\begin{equation}\label{resolution}
 \cdots \lrt \Hom_{\CC}(-, U_1)|_{\CU} \lrt \Hom_{\CC}(-, U_0)|_{\CU} \lrt \Hom_{\CC}(-, X)|_{\CU} \lrt 0.
\end{equation}
By applying the functor $\Hom(-, \CU(-,U))$, for $ U \in \CU$, on the above exact sequence we have the following commutative diagram of exact sequences
 \begingroup\scriptsize\[\xymatrix@C=0.4cm{ 0 \ar[r] & \Hom(\Hom_{\CC}(-, X)|_{\CU}, \CU(-,U)) \ar[r] \ar[d] & \Hom(\CU(-, U_0), \CU(-,U)) \ar[r] \ar[d] &\Hom(\CU(-, U_1), \CU(-,U)) \ar[d] & \\ 0 \ar[r] &\Hom_{\CC}(X,U) \ar[r] & \Hom_{\CC}(U_0, U) \ar[r] & \Hom_{\CC}(U_1, U) & }\] \endgroup
that implies an isomorphism $\Hom(\Hom_{\CC}(-,X)|_{\CU}, \CU(-, U)) \cong \Hom_{\CC}(X, U)$.

By the same method that used in \cite{AS2} for the module case one can deduce that for every object $X$ of ${}^{\perp_F}\CU$ there is a $\Hom_{\CC}(\CU, -)$-exact sequence
 \begin{equation*}
 	0 \lrt X \lrt U^0 \st{d^0}\lrt U^1 \st{d^1}\lrt U^2 \st{d^2}\lrt \cdots
 \end{equation*}
 with $\im d^i \in {}^{\perp_F}\CU$, for all $i\geq 0$,
 that for every object $X'$ of ${}^{\perp_F}\CU$, induces the commutative diagram
\begingroup\scriptsize\[\xymatrix@C=0.4cm{ 0 \ar[r] & \Hom(\Hom_{\CC}(-, X')|_{\CU}, \Hom_{\CC}(-,X)|_{\CU}) \ar[r] \ar[d] & \Hom(\CU(-, X'), \Hom_{\CC}(-,U^0)|_{\CU}) \ar[r] \ar[d] &\Hom(\CU(-, X'), \Hom_{\CC}(-,U^1)|_{\CU}) \ar[d] & \\
0 \ar[r] &\Hom_{\CC}(X',X) \ar[r] & \Hom_{\CC}(X', U^0) \ar[r] & \Hom_{\CC}(X', U^1) & }\]\endgroup
of exact sequences.
Hence, for every two objects $X$ and $X'$ of ${}^{\perp_F}\CU$
\[ \Hom (\Hom_{\CC}(X,-)|_{\CU}, \Hom_{\CC}(X',-)|_{\CU}) \cong \Hom_{\CC}(X', X).\]
A similar argument as above can be applied to see that
\[\Ext^i_{\CU}(\Hom_{\CC}(-,X)|_{\CU}, \CU(-,U)) \cong \Ext^i_F(X,U),\]
for every $\CC$-module $X$, every object $U$ of $\CU$ and all $i>0$. This implies that the exact sequence \ref{resolution} remains exact after applying the Hom-functor $\Hom(-, \CU(-,U))$ for each object $U$ of $\CU$. Moreover, consider a $\Hom_{\CC}(\CU, -)$-exact sequence
\begin{equation}\label{coresolution}
	0 \lrt X \lrt U^0 \st{d^0}\lrt U^1 \st{d^1}\lrt U^2 \st{d^2}\lrt \cdots
\end{equation}
with $\im d^i \in {}^{\perp_F}\CU$, for all $i\geq 0$. Then we an exact sequence
\[0 \lrt \Hom_{\CC}(-, X) |_{\CU} \lrt \CU(-, U^0) \lrt \CU(-, U^1) \lrt \CU(-,U^2) \lrt \cdots\]
that remains exact after applying the Hom-functor $\Hom(-, \CU(-,U))$ for each object $U$ of $\CU$. Therefore, $\Phi(X) = \Hom_{\CC}(-, X)|_{\CU}$ belongs to $\Gprj \CU$.

Now, let $Y$ be an object of $\Gprj \CU$. Hence, we have an exact sequence
$$\xymatrix{ 0 \ar[r] & Y \ar[r] & \CU(-, U^0) \ar[r]^{\CU(-, f)} & \CU(-, U^1) }.$$
Set $K:= \Ker f$. There is an exact sequence
\[ \xymatrix{ 0 \ar[r] &\Hom_{\CC}(-, K)|_{\CU} \ar[r] & \CU(-, U^0) \ar[r]^{\CU(-, f)} & \CU(-, U^1) }\] that yields $Y = \Hom_{\CC}(-, K)|_{\CU}$.
Since $Y= \Hom_{\CC}(-, K)|_{\CU}$ belongs to $\Gprj \CU$, we deduce that $\Ext_{\CU}^i(\Hom_{\CC}(-, K)|_{\CU}, \CU(-, U))=0$ for all object $U$ in $\CU$ and all $i >0$. Thus $\Ext^i_F(K, U)=0$, for all $U \in \CU$ and all $i>0$ and so $K$ belongs to ${}^{\perp_F}\CU$. This means that the functor $\Phi$ is dense. Therefore, the proof is complete.
\end{proof}

\subsection{Induced Galois $G$-(pre)covering on $n$-minimal Auslander-Gorenstein categories}
Let $\CU$ be a $G$-equivariant $n$-precluster tilting subcategory of $\mmod \CC$. According to the proof of Theorem 3.8 of \cite{Gre}, the category $\mmod \CU$ is an $n$-minimal Auslander-Gorenstein category, as defined in Definition \ref{Def}. It will be shown in Proposition \ref{Prop4.5} that the push-down functor $P_*: \mmod \CC \to \mmod(\CC/G)$ induces a $G$-precovering functor $\mmod P_*: \mmod \CU \to \mmod (\CU/G)$. Furthermore, the restriction of this functor, $\mmod P_*|: \Gprj \CU \to \Gprj (\CU/G)$, is dense.
\begin{sdefinition}
Let $\CA$ be an abelian category. We say that $\CA$ has dominant dimension at least $n+1$ and denote it by ${\rm dom.dim } (\CA) \geq n+1$, if every projective object of $\CA$ has an injective resolution whose first $n$ terms are projective objects.
\end{sdefinition}

\begin{sdefinition}\label{Def}
An abelian category $\CA$ is called an $n$-minimal Auslander-Gorenstein category if it satisfies the following conditions:
 \begin{itemize}
 \item[$(i)$] ${\rm dom.dim}(\CA) \geq n + 1$,
 \item[$(ii)$] every projective object has an injective dimension less than or equal to $n + 1$.
 \end{itemize}
\end{sdefinition}

For any two objects $x$ and $y$ in a locally bounded $k$-category $\CC$, let $\CR(\CC(x,y))$ represent the subspace of $\CC(x,y)$ formed by the non-invertible morphisms. A locally bounded $k$-category $ \CC $ is defined as square-free if, for every pair of objects $ x $ and $ y $ in $ \CC $, the quotient $\CR(\CC(x,y))/\CR^2(\CC(x,y))$ has a $ k $-dimension less than or equal to 1. It is known that if $\CC$ is a locally bounded $G$-category, then $\CC/G$ is so, see \cite[Proposition 1.2]{G}. It can be readily verified that if $ \CC $ is a locally bounded $G$-category and $ \CC/G $ is square-free, then $ \CC $ is also square-free. For further details, see \cite[Section 3]{BG}.

Bongartz and Gabriel demonstrated that if $ \CC $ is a locally bounded $G$-category with a free $ G $-action on $ \ind \CC $ such that $ \CC/G $ is square-free, then the category $ \mmod \CC $ is an Auslander category if and only if $\mmod(\CC/G) $ is also an Auslander category \cite[Proposition 3.5]{BG}. The result presented below can be viewed as a generalization of this fact.

\begin{sproposition}\label{BonGab}
Let $\CC$ be a locally bounded $G$-category with free $G$-action on $\ind \CC$ such that $\CC/G$ is square-free. Then $\mmod \CC$ is an $n$-minimal Auslander-Gorenstein category if and only if so is $\mmod(\CC/G)$.
\end{sproposition}

\begin{proof}
First, note that according to Proposition 3.4 of \cite{BG}, a $\CC$-module $ M $ in $\mmod \CC$ is injective if and only if $ P_*(M) $ is injective as well.

Assuming that $\mmod(\CC/G)$ is $ n $-minimal Auslander-Gorenstein, consider a projective $\CC$-module $ \CC(-, x) $ along with an injective resolution:
\[0 \to \CC(-, x) \to I^0 \to I^1 \to \cdots \to I^{n-1} \to I^n \to \cdots.\]

Since the push-down functor is exact, this gives us an injective resolution:
\[0 \to P_*(\CC(-, x)) \to P_*(I^0) \to P_*(I^1) \to \cdots \to P_*(I^{n-1}) \to P_*(I^n) \to \cdots\]
of the projective object $ P_*(\CC(-, x)) \cong \CC/G(-, Px) $. By assumption, the first $ n $ terms of this resolution are projective objects. Thus, the first $ n $ terms of the injective resolution of $ \CC(-, x) $ are also projective, thanks to Proposition 3.2 of \cite{BG}. This implies that the dominant dimension of $\mmod \CC$ is at least $ n+1 $.

Next, we show that every projective $\CC$-module $ P = \CC(-, x) $ has an injective dimension less than or equal to $ n+1 $. To demonstrate this, it suffices to verify that $\Ext^i_{\CC}(-, P) = 0$ for all $ i > n+2 $.

By Lemma \ref{DILemma}, for each $\CC$-module $ M $ in $\mmod \CC$,
\[\Ext_{\CC/G}^i(P_*(M), P_*(P)) \cong \bigoplus_{a \in G} \Ext^i_{\CC}(M, {}^a P).\]

By assumption, $\Ext_{\CC/G}^i(P_*(M), P_*(P)) = 0$ for all $ i > n+2 $. Therefore, $\bigoplus_{a \in G} \Ext^i_{\CC}(M, {}^a P)$ vanishes, which implies that its direct summand $\Ext^i_{\CC}(M, P)$ also vanishes for $ i > n+2 $. This completes the proof of the `if' direction.

A similar argument can be made to show the `only if' part.
\end{proof}

An analogue of the higher Auslander-Solberg correspondence for exact categories was stated and proved in \cite{Gre}. If we consider the abelian category $\mmod \CC$ as an exact category, we can directly deduce the following theorem from \cite[Theorem 3.8]{Gre}.

\begin{stheorem}
For every integer $ n \geq 0 $, the mapping $ \mathcal{U} \mapsto \text{mod} \, \mathcal{U} $ establishes a bijective correspondence between:
\begin{itemize}
 \item[$(i)$] the equivalence classes of $ n $-precluster tilting subcategories of $ \mmod\mathcal{C} $ for some locally bounded $ k $-category $ \mathcal{C} $,
 \item[$(ii)$] the equivalence classes of $ n $-minimal Auslander-Gorenstein categories.
\end{itemize}
\end{stheorem}

Note that a $G$-action $A$ on the category $\CC$ can be extended to a $G$-action $\mmod \bar{A}$ on the category $\mmod (\mmod \CC)$. Specifically, for each object $T$ in $\mmod (\mmod \CC)$ and each element $a$ of $G$, we define ${}_a T := (\mmod \bar{A}) (a) (T)$ as follows: for every finitely presented $\CC$-module $X$ and every morphism $f$, we set ${}_a T(X) = T({}^{a^{-1}} X)$ and ${}_a T(f) = T({}^{a^{-1}} f)$. Here, ${}^{a^{-1}} X$ and ${}^{a^{-1}} f$ are determined by the action of $G$ on $\mmod \CC$, as introduced in section \ref{action}.

According to \cite[Proposition 5.26]{HAK}, there exists a functor
\[\mmod P_*: \mmod (\mmod \CC) \to \mmod (\mmod(\CC/G)).\]
In fact, let
\[T= \text{Coker}(\Hom_{\CC}(-,X) \xrightarrow{\Hom_{\CC}(-,f)} \Hom_{\CC}(-,Y)),\]
then we have
\[\mmod P_*(T) = \text{Coker}(\Hom_{\CC}(-,P_* X) \xrightarrow{\Hom_{\CC}(-,P_* f)} \Hom_{\CC}(-,P_* Y)).\]

\begin{sproposition}\label{Prop4.5}
Let $\CC$ be a locally bounded $G$-category with free $G$-action on $\ind \CC$.
\begin{itemize}
\item[(i)] If $\CU$ is a $G$-equivariant $n$-precluster tilting subcategory of $\mmod \CC$ such that $P_{*}(\CU)$ is functorially finite in $\mmod(\CC/G)$, then $\mmod P_*: \mmod \CU \lrt \mmod (\CU/G)$ is a $G$-precovering of $n$-minimal Auslander-Gorenstein categories.
\item[(ii)] If $\CC$ is, in addition, locally support-finite and $\CU$ is a $G$-equivariant $n$-precluster tilting subcategory of $\mmod \CC$, then $\mmod P_*: \mmod \CU \lrt \mmod (\CU/G)$ is a Galois $G$-precovering of $n$-minimal Auslander-Gorenstein categories that its restriction \\ $\mmod P_*|: \Gprj \CU \lrt \Gprj (\CU/G)$ is dense.
\end{itemize}
\end{sproposition}

\begin{proof}
$(i)$ This follows directly from Theorem \ref{Main1} and \cite[Proposition 5.26]{HAK}.
	
$(ii)$ We have the following commutative diagram
\[\xymatrix{
		\Gprj \CU \ar[rr]^{\mmod P_*|} && \Gprj (\CU/G) \\
		\CZ(\CU) \ar[u]^{\Phi} \ar[rr]^{P_*}&& \CZ(\CU/G) \ar[u]^{\Phi}
	}\]
such that $\Phi$ is an equivalence of categories, by Proposition \ref{equivalence}. Since by Proposition \ref{Z(U)}, $P_*: \CZ(\CU) \lrt \CZ(\CU/G)$ is dense, $\mmod P_*|: \Gprj \CU \lrt \Gprj (\CU/G)$ is also a dense functor as well.
\end{proof}

\begin{scorollary}
	Let $\CC$ be a locally support-finite $G$-category with free $G$-action on $\ind \CC$. If $\CU$ is a $G$-equivariant $n$-precluster tilting subcategory of $\mmod \CC$, then there is the following commutative diagram
	\[\xymatrix{& \CZ(\CU) \ar[rr]^{P_*}\ar[ld]_{\simeq} && \CZ(\CU/G) \ar[ld]_{\simeq}\\
		\Gprj \CU \ar[rr]^{\mmod P_*|} \ar[dd] && \Gprj \CU/G \ar[dd]
		\\
		& \CU \ar[rr]^{P_*} \ar@{^(->}[uu] \ar@{_(->}[ld] && \CU/G \ar@{^(->}[uu] \ar@{_(->}[ld] \\
		\mmod \CU \ar[rr]^{\mmod P_*} && \mmod \CU/G }\]
	with Galois $G$-precovering rows such that $P_*: \CU \lrt \CU/G$, $P_*: \CZ(\CU) \lrt \CZ(\CU/G) $ and $\mmod P_*|: \Gprj \CU \lrt \Gprj (\CU/G)$ are dense.
\end{scorollary}

\section{Induced Galois $ G$-(pre)covering on support tilting pairs}\label{tau_nTilting}
A full subcategory $\CU$ of $\mmod \CC$ is defined as $n$-cluster tilting if it is functorially finite in $\mmod \CC$ and satisfies the condition:
$${}^{\perp_{n-1}} \CU = \CU = \CU^{\perp_{n-1}}.$$

It has been established in \cite[Theorem 2.14]{DI} that if $\CC$ is a locally bounded $G$-category with a free $G$-action on $\ind \CC$, and if $\CM$ is a $G$-equivariant $n$-cluster tilting subcategory of $\mmod \CC$ such that $P_*(\CM)$ is a functorially finite subcategory of $\mmod(\CC/G)$, then $P_*(\CM)$ is an $n$-cluster tilting subcategory of $\mmod(\CC/G)$.

Zhou and Zhu, in \cite{ZZ}, introduced $\tau_n$-tilting objects for every projectively generated $n$-abelian category. This is equivalent to being an $n$-cluster tilting subcategory of an abelian category that has enough projectives.

\begin{definition}\label{t_n tilt}
Let $\CC$ be a locally bounded $G$-category and $\CM$ be a $G$-equivariant $n$-cluster tilting subcategory of $\mmod \CC$.
	
\begin{itemize}
	\item[$(i)$] An object $M$ of $\CM$ is called $(G, \tau_n)$-rigid, if $\Hom_{\CC}(M , {}^a \tau_n M)=0$ for all $a \in G$.
	\item[$(ii)$] A pair $(M, P)$ in $\CM$ with $P$ projective is called $(G, \tau_n)$-rigid pair, if $M$ is $(G, \tau_n)$-rigid and $\Hom_{\CC}(P, {}^a M)=0$ for all $a \in G$.
	\item[$(iii)$] A $(G, \tau_n)$-rigid pair $(M, P)$ in $\CM$ is called support $(G, \tau_n)$-tilting pair, if it satisfies the following two conditions
	\begin{enumerate}
		\item If $(M\oplus N, P)$ is a $(G, \tau_n)$-rigid pair for some $N \in \CM$, then there is a finite subset $H$ of $G$ such that $N \in \add \oplus_{a \in H} {}^a M$.
		\item For every indecomposable projective $\CC$-module $Q$, $Q \in \add {}^a P$ for some $a \in G$ if and only if $\Hom_{\CC}(Q, {}^a M)=0$ for all $a \in G$.
	\end{enumerate}
\end{itemize}
\end{definition}

\begin{sproposition}\label{tau_nTilt}
Let $\CC$ be a locally bounded $G$-category with a free $G$-action on $\ind \CC$. Assume that $\CM$ is an $n$-cluster tilting subcategory of $\mmod \CC$ such that $P_*(\CM)$ is an $n$-cluster tilting subcategory of $\mmod(\CC/G)$. If $(M,P)$ in $\CM$ is a support $(G, \tau_n)$-tilting pair, then $(P_*(M), P_*(P))$ is a support $\tau_n$-tilting pair in $P_*(\CM)$.
\end{sproposition}

\begin{proof}
First observe that in view of the proof of Theorem \ref{P_n(C)}, $P_*(\tau_n M) \cong \tau_n P_*(M)$. Therefore there exist the isomorphisms
\[ \begin{array}{ll}
		\Hom_{\CC/G}(P_*(M), \tau_n P_*(M)) & \cong \Hom_{\CC/G}(P_*(M), P_*(\tau_n M))\\
		& \cong \oplus_{a \in G} \Hom_{\CC}(M, {}^a \tau_n M)
	\end{array}\]
that imply the $ \tau_n$-rigidity of $P_*(M)$. The same argument show that $(P_*(M), P_*(P))$ is a $\tau_n$-rigid pair.
	
Let there exist an object $P_*(N)$ of $P_*(\CM)$ such that $(P_*(M) \oplus P_*(N), P_*(P))$ is a $\tau_n$-rigid pair. There are the following isomorphisms
\[ \begin{array}{lll}
	\Hom_{\CC/G}(P_*(M) \oplus P_*(N), \tau_n(P_*(M) \oplus P_*(N))) & \cong \Hom_{\CC/G} (P_*(M\oplus N), \tau_n(P_*(M \oplus N)))\\
	& \cong \Hom_{\CC/G} (P_*(M\oplus N), P_*(\tau_n(M \oplus N)))\\
	& \cong \oplus_{a \in G} \Hom_{\CC}(M\oplus N , {}^a \tau_n (M \oplus N))	\end{array}\]
and
\[ \begin{array}{lll}
			\Hom_{\CC/G}(P_*(P) , P_*(M) \oplus P_*(N)) & \cong \Hom_{\CC/G} (P_*(P), P_*(M\oplus N)))\\
			& \cong \oplus_{a \in G} \Hom_{\CC}(P, {}^a (M \oplus N)). \end{array}\]

That is a pair $(M \oplus N, P)$ is a $(G, \tau_n)$-rigid pair in $\mmod \CC$. Hence, there is a finite subset $H$ of $G$ such that $N$ lies in $\add (\oplus_{a \in H} {}^a M)$. This, in turn, implies that $P_*(N)$ belongs to $\add P_*(M)$.
		
Finally, consider a finitely generated projective $\CC/G$-module ${\bar Q}$. Let $\Hom_{\CC/G}(\bar{Q}, P_*(M))=0$. Since every finitely generated projective $\CC/G$-module can be written as finite direct sums of $\CC/G(-,Px)$, where $x$ runs through objects of $\CC$, there is a finitely generated projective $\CC$-module $Q$ such that $P_*(Q)\cong \bar{Q}$. Therefore, we have an isomorphism
\[ \Hom_{\CC/G} (\bar{Q}, P_*(M))\cong \oplus_{a \in G} \Hom_{\CC}(Q, {}^a M)\]
that implies $\Hom_{\CC}(Q, {}^a M)=0$, for all $a \in G$. By definition, $Q$ lies in $\add {}^a P$ for some $a \in G$. Since the push-down functor $P_*$ preserves direct sums, $P_*(Q) \cong \bar{Q}$ belongs to $\add P_*(P)$.
Conversely, it is plain that $\Hom_{\CC/G}(\bar{Q}, P_*(M))=0$ for every $\bar{Q}\in \add P_*(P)$. So, we deduce that $(P_*(M), P_*(P))$ is a support $\tau_n$-tilting pair in $\mmod(\CC/G)$.
\end{proof}

\begin{corollary}\label{T_n-tiltpreserves}
Let $\CC$ be a locally bounded $G$-category with a free $G$-action on $\ind \CC$.
\begin{itemize}
\item[$(i)$] Assume that $\CM$ is a $G$-equivariant $n$-cluster tilting subcategory of $\mmod \CC$ such that $P_*(\CM)$ is a functorially finite subcategory of $\mmod(\CC/G)$. If $(M,P)$ in $\CM$ is a support $(G, \tau_n)$-tilting pair, then $(P_*(M), P_*(P))$ is a support $\tau_n$-tilting pair in $P_*(\CM)$.
		
\item[$(ii)$] Assume that, in addition, $\CC$ is locally support-finite and $\CM$ is an $n$-cluster tilting subcategory of $\mmod \CC$. If $(M,P)$ in $\CM$ is a support $(G, \tau_n)$-tilting pair, then $(P_*(M), P_*(P))$ is a support $\tau_n$-tilting pair in $P_*(\CM)$.
\end{itemize}
\end{corollary}

\begin{proof}
$(i)$ It follows from Theorem 2.14 of \cite{DI} that $P_*(\CM)$ is an $n$-cluster tilting subcategory of $\mmod(\CC/G)$. The result follows from Proposition \ref{tau_nTilt}.

$(ii)$ Note that $P_*(\CM)$ is functorially finite thanks to Lemma \ref{approximation}. Now, $(i)$ yields the result.
\end{proof}

\begin{sproposition}\label{PullupT_ntilt}
Let $\CC$ be a locally bounded $G$-category with a free $G$-action on $\ind \CC$. Let $\CM$ be a $G$-equivariant subcategory of $\mmod \CC$ such that $P_*(\CM)$ is an $n$-cluster tilting subcategory of $\mmod(\CC/G)$. If $M$ is an object of $\CM$ such that $(P_*(M), P_*(P))$ is a support $\tau_n$-tilting pair in $P_*(\CM)$, then $(M, P)$ is a support $(G, \tau_n)$-tilting pair in $\CM$.
\end{sproposition}

\begin{proof}
The same argument as in Proposition \ref{tau_nTilt} together with the $k$-isomorphism in Definition \ref{k-iso} implies that $(M, P)$ is a $(G, \tau_n)$-rigid pair.
	
Assume that there is an object $N \in \CM$ such that $(M\oplus N, P)$ is a $(G, \tau_n)$-rigid pair. An standard argument, in conjunction with $k$-isomorphism in Definition \ref{k-iso}, can be applied to get that $(P_*(M\oplus N), P_*(P))$ is a $\tau_n$-rigid pair in $P_*(\CM)$. Hence, $P_*(N) \in \add P_*(M)$. Since the push-down functor $P_*$ preserves direct summand, we may assume that $P_*(N) \cong \oplus_I P_*(M)$, where $I$ is a finite index set. The same argument as in Corollary \ref{tau_nSelfinj} works to show that
\[\begin{array}{ll}
		N & \cong \oplus_{a \in H} {}^a (\oplus_I M) \\
		& \cong \oplus_{a \in H} (\oplus_I {}^a M),
	\end{array}\]
for some finite subset $H$ of $G$. That is $N \in \add \oplus_{a \in H} {}^a M$.
	
Consider an indecomposable projective $\CC$-module $\CC(-,x)$, where $x$ is an object of $\CC$. If $\CC(-, x) \in \add {}^a P$ for some $a \in G$, then $P_*(\CC(-,x)) \cong \CC/G(-, Px)$ belongs to $\add P_*(P)$. So, $$\Hom_{\CC/G}(\CC/G(-, Px) , P_*(M))=0$$. By using the $k$-isomorphism in Definition \ref{k-iso}, we have $\Hom_{\CC}(\CC(-,x) , {}^a P)=0$, for all $a \in G$. For the converse, assume that $\Hom_{\CC}(\CC(-,x) , {}^a P)=0$ for all $a \in G$. Hence, $\Hom_{\CC/G}(\CC/G(-, Px) , P_*(P))=0$. So, by definition, $\CC/G(-, Px) \in \add P_*(P)$.
	
A similar argument as above should be applied to see that $\CC(-, x)$ is isomorphic to a direct summand of $\oplus_{a \in H} (\oplus_I {}^a P)$, for a finite subset $H$ of $G$ and a finite index set $I$. Since, $\CC(-, x)$ is indecomposable, $\CC(-,x)$ is isomorphic to a direct summand of ${}^a P$ for some $a \in G$. This completes the proof.
\end{proof}

Let $\CC$ be a locally bounded $G$-category with a free $G$-action on $\ind \CC$. Assume that $\CM$ is an $n$-cluster tilting subcategory of $\mmod \CC$ such that $P_*(\CM)$ is an $n$-cluster tilting subcategory of $\mmod(\CC/G)$.
Let $s(G, \tau_n)\mbox{-}\mathsf{tilt}(\CC)$, resp $s \tau_n\mbox{-}\mathsf{tilt}(\CC/G)$, denote the set of all support $(G, \tau_n)$-tilting $\CC$-modules in $\CM$, resp. all support $\tau_n$-tilting $\CC/G$-modules in $P_*(\CM)$.

\begin{sproposition}\label{Prop-tau_n}
Let $\CC$ be a locally bounded $G$-category with a free $G$-action on $\ind \CC$. Then the push-down functor $P_*: s(G, \tau_n)\mbox{-}\mathsf{tilt}(\CC) \lrt s \tau_n\mbox{-}\mathsf{tilt}(\CC/G)$ is a $G$-precovering. Moreover, if $\CC$, in addition, is locally support-finite, then $P_*: s(G, \tau_n)\mbox{-}\mathsf{tilt}(\CC) \lrt s \tau_n\mbox{-}\mathsf{tilt}(\CC/G)$ is dense.
\end{sproposition}

\begin{proof}
In view of Corollary \ref{T_n-tiltpreserves}, the push-down functor $P_*: \mmod \CC \lrt \mmod(\CC/G)$ can be restricted to the functor $P_*: s(G, \tau_n)\mbox{-}\mathsf{tilt}(\CC) \lrt s \tau_n\mbox{-}\mathsf{tilt}(\CC/G)$. This, in turn, implies that $P_*: s(G, \tau_n)\mbox{-}\mathsf{tilt}(\CC) \lrt s \tau_n\mbox{-}\mathsf{tilt}(\CC/G)$ is a $G$-precovering.

Moreover, assume that $G$ is locally support-finite and consider an object $X$ of $s \tau_n\mbox{-}\mathsf{tilt}(\CC/G)$. By Theorem \ref{loc-supp-fin}, there is a finitely generated $\CC$-module $M$ such that $P_*(M) \cong X$. Now, Proposition \ref{PullupT_ntilt} yields that $M$ lies in $s(G, \tau_n)\mbox{-}\mathsf{tilt}(\CC)$ and the proof is completed.
\end{proof}

For a locally bounded $G$-category $\CC$, let $(G, \tau_n)\mbox{-} \mathsf{rig}(\CC)$, resp. $\tau_n\mbox{-}\mathsf{rig}(\CC/G)$, denote the set of all $(G, \tau_n)$-rigid $\CC$-modules, resp $\tau_n$-rigid $\CC/G$-modules.

\begin{lemma}\label{CoverRig}
Let $\CC$ be a locally bounded $G$-category with a free $G$-action on $\ind \CC$. Then the push-down functor $P_*: (G, \tau_n)\mbox{-}\mathsf{rig}(\CC) \lrt \tau_n\mbox{-}\mathsf{rig}(\CC/G)$ is a $G$-precovering. Moreover, if $\CC$, in addition, is locally support-finite, then $P_*: (G, \tau_n)\mbox{-}\mathsf{rig}(\CC) \lrt \tau_n\mbox{-}\mathsf{rig}(\CC/G)$ is a $G$-covering.
\end{lemma}

\begin{proof}
	The same argument as in the proof of Proposition \ref{Prop-tau_n} works to prove the following result.
\end{proof}

\begin{definition}
A locally bounded $k$-category $\CC$ is called locally $(G, \tau_n)$-tilting finite, resp. locally $\tau_n$-tilting finite, if for every object $x$ of $\CC$ there exist only finitely many indecomposable $(G, \tau_n)$-rigid, resp. $\tau_n$-rigid, $\CC$-module $M$ with $M(x)\neq 0$.
\end{definition}
The following result can be considered as a generalization of Theorem B \cite{PPW}.
\begin{corollary}\label{Corr}
Let $\CC$ be a locally support-finite $G$-category with free $G$-action on $\ind \CC$. Then the push-down functor induces a Galois $G$-covering
\[P_*: \ind ((G, \tau_n)\mbox{-}\mathsf{rig}(\CC)) \lrt \ind ( \tau_n\mbox{-}\mathsf{rig}(\CC/G)). \]
In particular,
 $\CC$ is locally $(G, \tau_n)$-tilting finite if and only if $\CC/G$ is locally $\tau_n$-tilting finite.
\end{corollary}

\begin{proof}
Since $P_*: (G, \tau_n)\mbox{-}\mathsf{rig}(\CC) \lrt \tau_n\mbox{-}\mathsf{rig}(\CC/G)$ is $G$-precovering, by \cite[Lemma 2.6]{BL}, $P_*$ maps decomposable objects to decomposable objects. Hence, by \cite[Lemma 3.5]{G}, every object $X$ in $(G, \tau_n)\mbox{-}\mathsf{rig}(\CC)$ is indecomposable if and only if $P_*(X)$ is indecomposable in $\tau_n\mbox{-}\mathsf{rig}(\CC/G)$. Moreover, since $\CC$ is locally support-finite, by Proposition \ref{loc-supp-fin}, we can take
\[\ind(\tau_n\mbox{-}\mathsf{rig}(\CC/G) )= \{ P_*(X) \mid X \in \ind((G, \tau_n)\mbox{-}\mathsf{rig}(\CC))\}.\]
Now, Lemma \ref{CoverRig} implies that $P_*: \ind ((G, \tau_n)\mbox{-}\mathsf{rig}(\CC)) \lrt \ind ( \tau_n\mbox{-}\mathsf{rig}(\CC/G))$ is a Galois $G$-covering. Hence , we have the isomorphism
\[\ind(\tau_n\mbox{-}\mathsf{rig}(\CC/G) ) \cong \ind((G, \tau_n)\mbox{-}\mathsf{rig}(\CC))/G \]
which implies the result.
\end{proof}

\section*{Acknowledgments}
The first author's research was supported by funding from the Iran National Science Foundation (INSF) under project No. 4001480. Part of this work was conducted during the second author's visit to the Institut des Hautes \'{E}tudes Scientifiques (IHES) in Paris, France. He expresses his sincere gratitude for the support and the enriching academic environment provided by IHES. The research of the last author was in part supported by a grant from IPM (No.1404180041).

\end{document}